\documentclass[12pt]{article}
\RequirePackage[OT1]{fontenc}
\RequirePackage{amsthm,amsmath}
\RequirePackage[numbers]{natbib}
\RequirePackage[colorlinks,citecolor=blue,urlcolor=blue]{hyperref}
\usepackage{amsfonts}
\usepackage{color}
\usepackage{pifont}
\usepackage{amssymb}
\usepackage[english]{babel}
\usepackage[T1]{fontenc}
\usepackage{graphicx}
\usepackage{subfigure}
\usepackage{latexsym}
\usepackage{amstext}
\usepackage{mathrsfs}
\usepackage{hyperref}
\usepackage{dsfont}
\usepackage{graphics}
\usepackage{enumerate}
\usepackage{paralist}
\usepackage{color}
\usepackage{fullpage}
\usepackage{wrapfig}

\newtheorem{prop}{Proposition}[section]

\newtheorem{lem}[prop]{Lemma}
\newtheorem{theo}[prop]{Theorem}

\numberwithin{equation}{section}

\newcommand{\beq}{\begin{eqnarray}}
\newcommand{\beqq}{\begin{eqnarray*}}
\newcommand{\eeq}{\end{eqnarray}}
\newcommand{\eeqq}{\end{eqnarray*}}

\usepackage{mathtools}
\usepackage[english]{babel}
\usepackage[utf8]{inputenc}
\usepackage{amsmath}
\usepackage{graphicx}
\usepackage[colorinlistoftodos]{todonotes}
\usepackage{polynom}
\usepackage{amsfonts}
\usepackage{dsfont}
\usepackage{amsthm}
\usepackage{hyperref}
\usepackage{graphicx}
\newtheorem{theorem}{Theorem}[section]

\newtheorem{lemma}{Lemma}[section]

\usepackage{fullpage}
\usepackage[T1]{fontenc}
\usepackage{hyperref}
\usepackage{xcolor}
\definecolor{link-color}{rgb}{0.15,0.4,0.15}
\hypersetup{
  colorlinks,
  linkcolor = {link-color}, citecolor = {link-color},
}

\usepackage{graphicx}
\usepackage{hyperref}
\usepackage{amsmath,amsthm,amssymb}
\usepackage{bbm}
\usepackage{tabularx}

\makeatletter
\def\namedlabel#1#2{\begingroup
    #2%
    \def\@currentlabel{#2}%
    \phantomsection\label{#1}\endgroup
}
\makeatother

\renewcommand{\textrm}[1]{\textup{#1}}%

    \def\d{{\textnormal d}}

\newenvironment{eqnarr}{\begin{IEEEeqnarray}{rCl}}{\end{IEEEeqnarray}\ignorespacesafterend}

\renewcommand{\eqref}[1]{\hyperref[#1]{(\ref*{#1})}}

\newcommand*{\norm}[1]{\lVert #1 \rVert}

    \def\beq{\begin{eqnarr}}
    \def\eeq{\end{eqnarr}}
    \def\beqq{\begin{eqnarray*}} 
    \def\eeqq{\end{eqnarray*}} 

        \def\d{{\rm d}}

    \def\d{{\textnormal d}}

\newcommand*{\pref}[1]{\hyperref[#1]{(\ref*{#1})}}
\newcommand*{\refpref}[2]{\hyperref[#2]{\ref*{#1}(\ref*{#2})}}

  \newcommand{\D}{{\rm d}}

\newcommand{\Q}{\texttt Q}

\newcommand{\eU}{\emph{\texttt U}}

\newcommand{\V}{\texttt V}

\newcommand{\ebS}{\emph{\texttt S}}

\newcommand{\ebF}{\emph{\texttt F}}

\newcommand{\bA}{\mathcal A}

\title{Stochastic Methods for Neutron Transport  Equation III: Generational many-to-one  and  
	$
	k_{\texttt{eff} } 
	$}
\author{A. M. G. Cox\footnote{
Department of Mathematical Sciences, University of Bath, Claverton Down, Bath, BA2 7AY, UK. Email: \texttt{a.g.m.cox@bath.ac.uk, e.l.horton@bath.ac.uk, a.kyprianou@bath.ac.uk}
} 
, \ E. L. Horton$^*$, A. E. Kyprianou$^*$,
\ D. Villemonais\footnote{Institut \'Elie Cartan de Lorraine,
Bureau 123,
Universit\'e de Lorraine,
54506,
Vandoeuvre-l\`es-Nancy Cedex,
France. Email: \texttt{denisvillemonais@gmail.com}
}
}

\date{\today}

\begin{document}
\maketitle

\begin{abstract}\hspace{0.1cm}
The Neutron Transport Equation (NTE) describes the flux of neutrons over time through an inhomogeneous fissile medium. In the recent articles~\cite{MultiNTE, SNTE},  a probabilistic solution of the NTE is considered in order to demonstrate a Perron-Frobenius type growth of the solution via its projection onto an associated leading eigenfunction. In \cite{SNTE-II, MCNTE}, further analysis is performed to understand the implications of this growth both in the stochastic sense, as well as from the perspective of Monte-Carlo simulation.

Such Monte-Carlo simulations are prevalent in industrial applications, in particular where regulatory checks are needed in the process of reactor core design. 
In that setting, however, it turns out that a different notion of growth takes centre stage, which is otherwise characterised by another eigenvalue problem. In that setting, the eigenvalue, sometimes called $k$-effective (written $k_{\texttt{eff}}$), has the physical interpretation as being the ratio of neutrons produced (during fission events) to the number lost (due to absorption in the reactor or leakage at the boundary) per typical fission event.

In this article, we aim to supplement \cite{MultiNTE, SNTE, SNTE-II, MCNTE}, by developing the stochastic analysis of the NTE further to the setting where a rigorous probabilistic interpretation of $k_{\texttt{eff}}$ is given, both in terms of a Perron-Frobenius type analysis as well as via classical operator analysis.

{\color{black}To our knowledge, despite the fact that an extensive engineering literature and industrial Monte-Carlo software is concentrated around the estimation of $k_{\texttt{eff}}$ and its associated eigenfunction, we believe that our work is the first rigorous treatment in the probabilistic sense (which underpins some of the aforesaid Monte-Carlo simulations).   }

\smallskip

\smallskip


\end{abstract}

\section{Introduction}\label{intro}
%

As described in~\cite{SNTE, SNTE-II, MultiNTE} the NTE is a balance equation for the flux of neutrons across a planar cross-section in an inhomogeneous fissile medium. The backwards form of the equation can be written as follows,
\begin{align}
\frac{\partial}{\partial t}\psi_t(r, \upsilon) &=\upsilon\cdot\nabla\psi_t(r, \upsilon)  -\sigma(r, \upsilon)\psi_t(r, \upsilon)\notag\\
&+ \sigma_{\texttt{s}}(r, \upsilon)\int_{V}\psi_t(r, \upsilon') \pi_{\texttt{s}}(r, \upsilon, \upsilon')\d\upsilon' + \sigma_{\texttt{f}}(r, \upsilon) \int_{V}\psi_t(r, \upsilon') \pi_{\texttt{f}}(r, \upsilon, \upsilon')\d\upsilon',
\label{bNTE}
\end{align}
where the flux $\psi_t(r, \upsilon)$ is a function of time, $t$ and the configuration variables $ (r, \upsilon) \in  D \times V$ where $D\subseteq\mathbb{R}^3$ is a non-empty, smooth, bounded convex domain such that $\partial D$ has zero Lebesgue measure, and $V = \{\upsilon\in \mathbb{R}^3:\upsilon_{\texttt{min}}\leq  |\upsilon|\leq \upsilon_{\texttt{max}}\}
$. Furthermore, the other components of~\eqref{bNTE} have the following interpretation:
\begin{align*}
\sigma_{\texttt{s}}(r, \upsilon) &: \text{ the rate at which scattering occurs from incoming velocity $\upsilon$,}\\
\sigma_{\texttt{f}}(r, \upsilon) &: \text{  the rate at which fission occurs from incoming velocity $\upsilon$,}\\
\sigma(r, \upsilon) &: \text{ the sum of the rates } \sigma_{\texttt{f}}+ \sigma_{\texttt{s}} \text{ and is known as the total cross section,}\\
\pi_{\texttt{s}}(r, \upsilon, \upsilon')\d\upsilon' &: \text{  the scattering yield at velocity $\upsilon'$ from incoming velocity }  \upsilon, \\
 &\hspace{0.5cm}\text{ satisfying }\textstyle{\int_V}\pi_{\texttt{s}}(r, \upsilon, \upsilon'){\rm d}\upsilon'=1,\text{ and }\\
 \pi_{\texttt{f}}(r, \upsilon, \upsilon')\d\upsilon' &:  \text{  the neutron yield at velocity $\upsilon'$ from fission with incoming velocity }   \upsilon,\\
 &\hspace{0.5cm}\text{ satisfying }{\color{black} \textstyle{\int_V}\pi_{\texttt{f}}(r, \upsilon, \upsilon')\d\upsilon' <\infty.}
 \end{align*}
We also enforce the following initial and boundary conditions
\begin{equation}
\left\{
\begin{array}{ll}
\psi_0(r, \upsilon) = g(r, \upsilon) &\text{ for }r\in D, \upsilon\in{V},
\\
&
\\
\psi_t(r, \upsilon) = 0& \text{ for } t \ge 0 \text{ and } r\in \partial D
\text{ if }\upsilon
\cdot{\bf n}_r>0,
\end{array}
\right.
\label{BC}
\end{equation}

where ${\bf n}_r$ is the outward unit normal at $r \in \partial D$ and $g: D \times V \to [0, \infty)$ is a bounded, measurable function. 
Throughout we will rely on the following assumptions in some (but not all) of our results:

{\bf 
{\color{black} 

\begin{itemize}
\item[(H1):]  Cross-sections $\sigma_{\texttt{s}}$, $\sigma_{\texttt{f}}$, $\pi_{\texttt{s}}$ and $\pi_{\texttt{f}} $ are uniformly bounded away from   infinity.

\item[(H2):] 
We have 
$
\sigma_{\texttt{s}} \pi_{\texttt{s}}  + 
\sigma_{\texttt{f}} \pi_{\texttt{f}}>0$ on $D\times V\times V$.

\item[(H3):]  There is an open ball $B$ compactly embedded in $D$ such that $\sigma_{\texttt{f}}\pi_{\texttt{f}} >0$ on $B\times V\times V$.

\item[(H4):]  the fission offspring are bounded in number  by the constant $N_{\texttt{max}}> 1$.
\end{itemize}
}
}

Note, the assumption (H1) ensures that all activity occurs at a maximum rate. Assumption (H2) ensures that at least some activity occurs, whether it be scattering or fission, together with (H3), it ensures that there is at least some fission as well as scattering. Finally (H4) is a physical constraint that is natural to nuclear fission, typically no more than 3 neutrons are produced during an average fission event. Figure \ref{fig} illustrates the complex nature of the in homogeneity in the domain one typically considers.
\bigskip

\begin{wrapfigure}{r}{0.5\textwidth}
\label{fig}
  \begin{center}
\includegraphics[height=8cm]{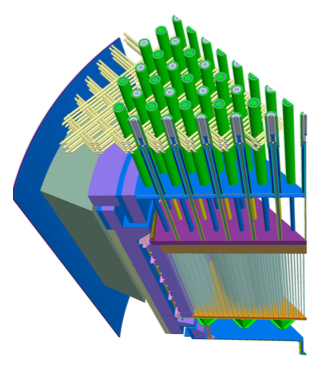}
  \end{center}
\caption{\it The geometry of a nuclear reactor core representing a physical domain $D$, on to which the different cross-sectional values of  $ \sigma_{\emph{\texttt{s}}},  \sigma_{\emph{\texttt{f}}},  \pi_{\emph{\texttt{s}}},  \pi_{\emph{\texttt{f}}}$ are mapped, also as a function of neutron velocity.}
\end{wrapfigure}

Due to the irregular nature of gradient operator, \eqref{bNTE} is meaningless in the pointwise sense, so it is often stated in one of two forms. The first is to treat \eqref{bNTE} as a weak linear partial integro-differential equation (PIDE) in an appropriate Banach  space, usually $L_2(D\times V)$, the space of functions  $f: D\times V\mapsto [0,\infty)$ which are finite with respect to the norm 
$
\norm{f}_2 = (\int_{D\times V}f(r,\upsilon)\d r\d\upsilon)^{1/2}
$); see e.g. \cite{D, DL6, M-K}. The second is to consider   the integral or {\it mild} form of \eqref{bNTE}. We refer the reader to~\cite{SNTE, SNTE-II, MultiNTE} and the references therein for a discussion on the various formulations of the NTE and its solution. We will also elaborate on both in  the forthcoming discussion.

\bigskip

For both formats of \eqref{bNTE}, the papers \cite{D, DL6, M-K, SNTE, MultiNTE} dealt with the time-eigenvalue problem and an associated Perron-Frobenius decomposition. More precisely, they give a rigorous stochastic meaning to the asymptotic 
\begin{equation}
\psi_t\sim {\rm e}^{\lambda_*t} c_g\varphi + o({\rm e}^{\lambda_*t}),
\label{psitilde}
\end{equation}
as $t\to\infty$, where $\lambda_*$ and $\varphi$ are the leading eigenvalue and associated eigenfunction associated to the NTE in the appropriate sense and $c_g$ is a constant that depends on the initial data $g$. 

\bigskip 

Such an understanding is important as it promotes a number of different Monte-Carlo algorithms that can be used to estimate both the lead eigenvalue $\lambda_*$ and the associated non-negative eigenfunction $\varphi$. The latter can be formulated as an eigenpair in $L_2(D\times V)$ satisfying
\begin{equation}
({\mathcal T} + \mathcal{S} +{\mathcal F})\varphi= \lambda_* \varphi,
\label{lambda}
\end{equation}
on $D\times V$, where 
 
\begin{equation}
\left\{
\begin{array}{rl}
{\mathcal T}{f}(r, \upsilon)  &:=  \upsilon\cdot\nabla{f}(r, \upsilon) - \sigma(r, \upsilon)\\
\\
{\mathcal S}{f}(r, \upsilon) &:= \sigma_{\texttt{s}}(r, \upsilon)\int_{V}{f}(r, \upsilon') \pi_{\texttt{s}}(r, \upsilon, \upsilon') \d\upsilon' \\
\\
{\mathcal F}{f}(r, \upsilon) &: =  \sigma_{\texttt{f}}(r, \upsilon) \int_{V}{f}(r, \upsilon') \pi_{\texttt{f}}(r, \upsilon, \upsilon')\d\upsilon',
\end{array}
\right.
\end{equation}
\smallskip
Here, we can think of $\lambda_*$ as characterising the rate of growth of flux in the system over time.
\bigskip

It turns out that, predominantly in industrial, engineering and  (some) physics literature, there is another eigenvalue problem that plays a fundamental role in the design and safety of nuclear reactors; see for example  Section 1.5 of \cite{LM}. The aforesaid eigenvalue problem involves finding (in any appropriate sense) an eigenpair  $k$ and $\phi$ such that 
\begin{equation}
({\mathcal T} + \mathcal{S})\phi + \frac{1}{k}{\mathcal F}\phi =0.
\label{keval2}
\end{equation}

%

The leading eigenvalue, which in the nuclear regulation industry is  called $k$-effective, written  $k_{\texttt{eff}}$,  has the physical interpretation as being the ratio of neutrons produced (during fission events) to the number lost (due to absorption in the reactor or leakage at the boundary). Another interpretation of $k$ is that it represents the average number of neutrons produced per fission event. It is this second interpretation which we exploit, since  $k_{\texttt{eff}}$ acts as a measure of neutrons produced between fission generations.

\bigskip

It is worth noting that the two eigenproblems offer potentially different sets of solutions,  however, they agree in terms of criticality. More precisely, in \eqref{lambda}, the triple $(\mathcal{T}, \mathcal{S}, \mathcal{F})$ is called critical if the leading eigenvalue, $\lambda_*$, in \eqref{lambda} is zero, and otherwise called subcritical (resp. supercritical) if $\lambda_*<0$ (resp. $\lambda_*>0$). In the setting of \eqref{keval2}, the triple $(\mathcal{T}, \mathcal{S}, \mathcal{F})$ is called critical if $k_{\texttt{eff}}=1$ and subcritical (resp. supercritical) if $k_{\texttt{eff}}<1$ (resp. $k_{\texttt{eff}}>1$). 

We note however that in \cite{fbrown}, there is a relationship between the two eigenvalues, regardless of the criticality of the system and at criticality, both \eqref{lambda} and \eqref{keval2} agree.

\bigskip

The main objective of this paper is to put into a rigorous setting the existence of the `leading' solutions to \eqref{keval2} in the two main contexts that the NTE \eqref{bNTE} is understood; that is, the weak linear PIDE context and the probabilistic context. Moreover, in the mild setting,  we will build an expectation semigroup, say  $(\Psi_n, n\geq 0)$, out of a stochastic process such that 
\[
\Psi_n[g]\sim k_{\texttt{eff}}^{-n} C_g\phi + o(k_{\texttt{eff}}^{-n}),
\]
for $g\in L^+_\infty(D\times V)$, as $n\to\infty$, and an appropriate choice of $C_g\geq 0$. (See Theorem \ref{7CVtheoremBis} below.)
This also provides a mathematically rigorous underpinning for many of the Monte-Carlo algorithms that are used in industry for computing $k_{\texttt{eff}}$. We will offer further discussion in this direction at the end of the paper.

\bigskip

The rest of this article is organised as follows. In the next section, we formally introduce the description of \eqref{bNTE} as a PIDE on a functional space, that is, we describe it as an abstract Cauchy problem. Moreover, we introduce two underlying stochastic processes, both of which can be used to describe the solution to the mild NTE. Also in this section, we introduce a second mild equation, \eqref{mildNTE}, whose eigen-solutions give us a sense in which we can characterise solutions to \eqref{keval2}.

\bigskip

In Section \ref{QSD}, we provide a solution to the newly introduced mild equation \eqref{mildNTE}. In addition, we state the main result of this paper (Theorem \ref{CVtheorem}) which shows the existence of a lead eigensolution to \eqref{mildNTE}.

\bigskip

In Section \ref{Sclassical}, for comparison, we show how to construct and give meaning to the lead eigensolution to \eqref{keval2}
in the setting of a functional space. In addition, we show how the two notions of the lead eigensolution, in this and the previous section, agree.

\bigskip 

In Section \ref{longproof}, we give the proof of the main result of Section \ref{QSD}. Finally, we conclude in Section \ref{discussion} with some discussion concerning the relevance of such results to previous work and Monte-Carlo methods.  


\section{Formulations of the NTE and associated eigenfunctions}
As alluded to in the introduction, there are two principal ways in which the NTE is formulated. In this section, we will elaborate on them in a little more mathematical detail for later convenience and context of our main results. 


\subsection{Abstract Cauchy Problem (ACP)} Following e.g. \cite{D, DL6, M-K}, we want to formulate \eqref{bNTE} in the function space $L_2(D\times V)$.
The so-called (initial-value) abstract Cauchy problem (ACP)
takes the form 
\begin{equation}
\dfrac{\partial{{u}}_t}{\partial t} = \bA{{u}}_t
\quad \text{ and }\quad
{{u}}_0 = g, 
\label{ACP}
\end{equation}
where $\bA = {\mathcal T} + {\mathcal S}+{\mathcal F}$ and  ${{u}}_t$ belongs to   the space $L_2({D}\times V)$, for $t\geq 0$ (in particular $g\in L_2({D}\times V)$).  Specifically,  $({{u}}_t, t\geq 0)$ is continuously differentiable in the space $L_2({D}\times V)$, meaning there exists a $\dot{{u}}_t\in  L_2({D}\times V)$, which is time-continuous in $ L_2({D}\times V)$ with respect to $\norm{\cdot}_2$ and such that  $\lim_{h\to 0}h^{-1}({u}_{t+h} - {u}_t)= \dot{{u}}_t$ for all $t\geq 0$. Necessarily, the solution to \eqref{ACP} forms a $c_0$-semigroup\footnote{Recall that 
a $c_0$-semigroup $(\texttt{V}_t, t\geq0)$ also goes by the name of a strongly continuous semigroup and, in the present context, this means it has  has the properties that (i) $\texttt{V}_0 = {\rm Id}$, (ii) $\texttt{V}_{t+s}[g] = \texttt{V}_t[\texttt{V}_s[g]]$, for all $s, t\geq 0$, $g\in  L_2({D}\times V)$ and (iii) for all $g\in  L_2({D}\times V)$, $\lim_{h\to0}\norm{\texttt{V}_h[g] - g}_2 = 0$.}. 
Moreover, 
$
{\rm Dom}(\bA): =\{g\in L_2(D\times V) : 
\upsilon\cdot\nabla g \in   L_2(D\times V)
\text{ and }g|_{\partial D^+} =0
\}
$
is the domain of $\bA$ and ${{u}}_t\in {\rm Dom}(\bA)$ for all $t\geq 0$. 

\bigskip

\begin{theo}\label{thACP} {\color{black} Suppose (H1) holds.}
For $g\in  L_2({D}\times V),$ the unique solution to \eqref{ACP} is given by  $(\texttt{\emph V}_t, t\geq 0)$, the  $c_0$-semigroup  generated by $(\bA, {\rm Dom}(\bA))$, i.e. the orbit $\texttt{\emph V}_t[g] := \exp(t \bA)g$. 
\end{theo} 

In the ACP setting, the notion of an eigenpair $(\lambda, \varphi)$
is well formulated on $L_2(D\times V)$ via \eqref{lambda}.
Equivalently,
it means we are looking for $\varphi\in L^+_2(D\times V)$ and $\lambda $ such that $\V_t[\varphi] = {\rm e}^{\lambda t}\varphi$ on $L^+_2(D\times V)$, for all $t\geq 0$. The sense in which we mean that $\lambda$ is a `leading' eigenvalue roughly boils down it corresponding to the eigenvalue in the spectrum of the operator $\mathcal{A}$ on $L_2(D\times V)$ with the largest real part (and, as usual, it is real valued itself), and moreover, its associated eigenfunction $\varphi$ is non-negative. As such, one expects the existence of a non-negative left eigenfunction $\tilde \varphi$ (e.g. in the sense that $\langle\tilde\varphi, \V_t[g]\rangle = {\rm e}^{\lambda t} \langle \tilde\varphi, g\rangle$ for $t\geq 0$) such that 
\begin{equation}
\norm{{\rm e}^{-\lambda t}\V_t[g] -  \langle\tilde\varphi, g\rangle\varphi}_2 = o({\rm e}^{-\lambda t}),
\label{Vtilde}
\end{equation}
as $t\to\infty$. Here, we use the notation $\langle f, g\rangle = \int_{D\times V} f(r,\upsilon)g(r,\upsilon)\d r\d\upsilon$, so that $\norm{\cdot}_2 = \langle\cdot, \cdot\rangle^{1/2}$.
Precise results of this nature can be found in \cite{DL6, M-K, MultiNTE}.

\subsection{Neutron branching process (NBP) and the mild NTE}\label{psi1}
We recall the neutron branching process (NBP) defined in~\cite{SNTE}, which at time $t\ge 0$ is represented by a configuration of particles which are specified via their physical location and velocity in $D\times V$, say $\{(r_i(t), \upsilon_i(t)): i = 1,\dots , N_t\}$, where $N_t$ is the number of particles alive at time $t\ge 0$. The NBP is then given by the empirical distribution of these configurations,
\begin{equation}
X_t(A) = \sum_{i=1}^{N_t}\delta_{(r_i(t), \upsilon_i(t))}(A), \qquad A\in\mathcal{B}(D\times V), \;t\ge 0,
\label{Xt}
\end{equation}
where $\delta$ is the Dirac measure, defined on $\mathcal{B}(D\times V)$, the Borel subsets of $D\times V$.
\bigskip

The evolution of $(X_t, t\geq 0)$ is a stochastic process valued in the space of atomic measures
$
\mathcal{M}(D\times V): = \{\textstyle{\sum_{i = 1}^n}\delta_{(r_i,\upsilon_i)}: n\in \mathbb{N}, (r_i,\upsilon_i)\in D\times V, i = 1,\cdots, n\}
$
which evolves randomly as follows.

\bigskip

A particle positioned at $r$ with velocity $\upsilon$ will continue to move along the trajectory $r + \upsilon t$, until one of the following things happens. 
\begin{enumerate}
\item[(i)] The particle leaves the physical domain $D$, in which case it is instantaneously killed. 
\item[(ii)] Independently of all other neutrons, a scattering event occurs when a neutron comes in close proximity to an atomic nucleus and, accordingly, makes an instantaneous change of velocity. For a neutron in the system with position and velocity $(r,\upsilon)$, if we write $T_{\texttt{s}}$ for the random time that scattering may occur, then independently of any other physical event that may affect the neutron, 
$
\Pr(T_{\texttt{s}}>t) = \exp\{-\textstyle{\int_0^t} \sigma_{\texttt{s}}(r+\upsilon s, \upsilon){\rm d}s \}, $ for $t\geq0.$

When scattering occurs at space-velocity $(r,\upsilon)$, the new velocity $\upsilon'\in V$ is selected independently with probability $\pi_{\texttt{s}}(r, \upsilon, \upsilon')\d\upsilon'$. 
\item[(iii)] Independently of all other neutrons, a fission event occurs when a neutron smashes into an atomic nucleus. 
For a neutron in the system  with initial position and velocity $(r,\upsilon)$, if we write $T_{\texttt{f}}$ for the random time that scattering may occur, then, independently of any other physical event that may affect the neutron, 
$
\Pr(T_{\texttt{f}}>t) = \exp\{-\textstyle{\int_0^t} \sigma_{\texttt{f}}(r+\upsilon s, \upsilon){\rm d}s \},$ for $t\geq 0.
$

When fission occurs, the smashing of the atomic nucleus produces lower mass isotopes and releases a random number of neutrons, say $N\geq 0$, which are ejected from the point of impact with randomly distributed, and possibly correlated, velocities, say $\{\upsilon_i: i=1,\cdots, N\}$. The outgoing velocities are described by  the atomic random measure 
\begin{equation}
\label{PP}
\mathcal{Z}(A): = \sum_{i= 1}^{N } \delta_{\upsilon_i}(A), \qquad A\in\mathcal{B}(V).
\end{equation} 

If such an event occurs at location $r\in\mathbb{R}^d$ from a particle with incoming velocity $\upsilon\in{V}$, 
we denote by ${\mathcal P}_{(r,\upsilon)}$ the law of $\mathcal{Z}$.
The probabilities ${\mathcal P}_{(r,\upsilon)}$ are such that, for  $\upsilon'\in{V}$, for bounded and measurable $g: V\to[0,\infty)$,
\begin{align}
\int_V g(\upsilon')\pi_{\texttt{f}}(r, v, \upsilon')\d\upsilon' &= {\mathcal E}_{(r,\upsilon)}\left[\int_V g(\upsilon')\mathcal{Z}(\d \upsilon')\right]
=: {\mathcal E}_{(r,\upsilon)}[\langle g, \mathcal{Z}\rangle].
\label{Erv}
\end{align}
Note, the possibility that $\Pr(N = 0)>0$, which will be tantamount to neutron capture (that is, where a neutron slams into a nucleus but no fission results and the neutron is absorbed into the nucleus). 
\end{enumerate}

The NBP is thus parameterised by the quantities $\sigma_{\texttt s}, \pi_{\texttt s}, \sigma_{\texttt f}$ and the family of measures ${\mathcal P} =({\mathcal P}_{(r,\upsilon)} , r\in D,\upsilon\in V)$ and accordingly we refer to it as a $(\sigma_{\texttt s}, \pi_{\texttt s}, \sigma_{\texttt f}, \mathcal{P})$-NBP. It is associated to the NTE via the relation \eqref{Erv},  and, although a $(\sigma_{\texttt s}, \pi_{\texttt s}, \sigma_{\texttt f}, \mathcal{P})$-NBP is uniquely defined, a NBP specified by $(\sigma_{\texttt s}, \pi_{\texttt s}, \sigma_{\texttt f}, \pi_{\texttt f})$ alone is not.  Nonetheless, it is easy to show that for a given $\pi_{\texttt f}$,  a $(\sigma_{\texttt s}, \pi_{\texttt s}, \sigma_{\texttt f}, \mathcal{P})$-NBP always exists which satisfies \eqref{Erv}. See the discussion in Section 2 of \cite{SNTE}.

%
%
%


\bigskip

Define
\begin{equation}
\psi_t[g](r,\upsilon) : = \mathbb{E}_{\delta_{(r, \upsilon)}}[\langle g, X_t \rangle], \qquad t\geq 0, r\in \bar{D}, \upsilon\in{V},
\label{semigroup}
\end{equation}
where $\mathbb{P}_{\delta_{(r, \upsilon)}}$ is the law of $X$ initiated from a single particle with configuration $(r, \upsilon)$, and $g\in L^+_\infty(D\times V)$, the space of non-negative uniformly bounded measurable functions on  $D\times V$. Here we have made a slight abuse of notation (see $\langle \cdot,\cdot\rangle$ as it appears in \eqref{Erv}) and written $\langle g, X_t \rangle$ to mean $\textstyle{\int_{D\times V}} g(r,\upsilon)X_t(\d r,\d \upsilon)$. The following result was shown in~\cite{MultiNTE, SNTE, DL6, D}.

\begin{theo}\label{mildlemma} {\color{black}
Suppose (H1) and (H2) hold.} For $g\in L^+_\infty(D\times V)$, the space of non-negative and uniformly bounded measurable functions on $D\times V$, there exist  constants $C_1,C_2>0$ such that  $\psi_t[g]$, as given in \eqref{semigroup}, is uniformly bounded by $ C_1\exp(C_2 t)$, for all $t\geq 0$. Moreover, $(\psi_t[g], t\geq 0)$ is the unique solution, which is bounded in time, to the so-called mild equation
\begin{equation}
\psi_t[g] = \emph{\texttt{U}}_t[g] + \int_0^t \emph{\texttt{U}}_s[({\ebS} + {\ebF})\psi_{t-s}[g]]\d s, \qquad t\geq 0,
\label{mild}
\end{equation}
for which \eqref{BC} holds, where the deterministic evolution $
\eU_t[g]( r,\upsilon) = g( r+\upsilon t, \upsilon)\mathbf{1}_{\{t<\kappa^D_{r,\upsilon}\}}, t\geq 0,
$
with $
\kappa_{r,\upsilon}^{D} := \inf\{t>0 : r+\upsilon t\not\in D\},
$
represents the advection semigroup associated with a single neutron travelling at velocity $\upsilon$ from $r$ at $t=0$.
\end{theo}
In \cite{MultiNTE} the below result was shown, which demonstrates the context in which the mild solution to the NTE and the ACP can be seen to coincide.

\begin{theo}\label{thm:co-semi} {\color{black}
Suppose (H1) and (H2) hold.} If $g\in L^+_\infty(D\times V)$ and if $(\psi_t[g], t\geq 0)$ is understood as the solution to the mild equation \eqref{mild}, then for $t\geq 0$,  $\emph{\texttt{V}}_t[g] = \psi_t[g]$ on $ L^+_2(D\times V)$, i.e.
$\norm{{\texttt{\emph V}}_t[g] - \psi_t[g]}_2 = 0$.  
\end{theo}

 In the probabilistic setting, the meaning of \eqref{lambda} can be interpreted as looking for a pair $\lambda$ and $\varphi$ such that, pointwise on $D\times V$, $\psi_t[\varphi] = {\rm e}^{\lambda t}\varphi$, for $t\geq 0$. As alluded to in \eqref{psitilde}, we have a similar asymptotic to \eqref{Vtilde}, which isolates the eigenpair $(\lambda, \varphi)$ in its limit. The notion of `leading' in the probabilistic setting is less obvious, however, due to Theorem \ref{thm:co-semi}, the eigenpairs that emerge from the mild setting and the weak linear PIDE setting should in principle agree on $L_2(D\times V)$. This is discussed with greater precision in \cite{MultiNTE, SNTE}.

\subsection{Neutron random walk (NRW)}\label{psi2}
There is a second stochastic representation of the unique solution to \eqref{mild}, which makes use of the so-called  neutron random walk (NRW). This process takes values in $D\times V$ and is defined by its scatter rates, $\alpha(r,\upsilon)$, $r\in D, \upsilon\in V$, and scatter probability densities $\pi(r,\upsilon,\upsilon')$, $r\in D, \upsilon,\upsilon'\in V$. When issued with a velocity $\upsilon$, the NRW will propagate linearly with that velocity until either it exits the domain $D$, in which case it is killed, or at the random time $T_{\texttt{s}}$ a scattering occurs, where $\Pr(T_{\texttt{s}}>t) = \exp\{-\textstyle{\int_0^t}\alpha(r+\upsilon t, \upsilon){\rm d}s \}, $ for $t\ge 0.$ When the scattering event occurs at position-velocity configuration $(r,\upsilon)$, a new velocity $\upsilon'$ is selected with
probability $\pi(r,\upsilon,\upsilon')\d\upsilon'$. If we denote by
$(R,\Upsilon) = ((R_t, \Upsilon_t), t\ge 0)$, the position-velocity of the resulting continuous-time random walk on $D\times V$ with an additional cemetery state   for when it leaves the domain $D$, it is easy to show that $(R,\Upsilon)$ is a Markov process. We call the process $(R,\Upsilon)$ an $\alpha\pi$-NRW. 

\bigskip

Given a NBP defined by $\sigma_{\texttt s}$, $\sigma_{\texttt f}$, $\pi_{\texttt{s}}$ and $\mathcal{P}$, set
\begin{align}
\alpha(r,\upsilon)\pi(r, \upsilon, \upsilon') = \sigma_{\texttt{s}}(r,\upsilon)\pi_{\texttt{s}}(r, \upsilon, \upsilon') + \sigma_{\texttt{f}}(r,\upsilon) \pi_{\texttt{f}}(r, \upsilon, \upsilon')\qquad r\in D, \upsilon, \upsilon'\in V.
\label{alpha}
\end{align}
and
\begin{equation}
\beta(r,\upsilon)=\sigma_{\texttt{f}}(r,\upsilon)\left(\int_V\pi_{\texttt{f}}(r, \upsilon,\upsilon')\d\upsilon'-1\right).
\label{betadef}
\end{equation}
The following result, given in~\cite{MultiNTE}, gives the so-called {\it many-to-one} representation of solution to the NTE in the form \eqref{mild}. 
\begin{lemma}\label{timeM21} {\color{black}
Suppose (H1) and (H2) hold,}
 we have that 
\begin{equation}
\psi_t[g](r,\upsilon)  = \mathbf{E}_{(r,\upsilon)}\left[{\rm e}^{\int_0^t\beta(R_s, \Upsilon_s)\D s}g(R_t, \Upsilon_t) \mathbf{1}_{\{t < \tau^D\}}\right], \qquad t\geq 0,r\in D, \upsilon\in V,
\label{phi}
\end{equation}
is a second representation of the unique mild solution (in the sense of Theorem \ref{mildlemma}) of the NTE \eqref{mild}, where $\tau^D = \inf\{t>0 : R_t\not\in D\}$ and ${\bf P}_{(r, v)}$ for the law of the $\alpha\pi$-NRW starting from a single neutron with configuration $(r, \upsilon)$. 
\end{lemma}

\subsection{Neutron generational process (NGP)}
In order to solve the $k$-eigenvalue problem~\eqref{keval2}, it turns out that $(\psi_t, t\geq 0)$ and $(\phi_t, t\geq 0)$ are not the right objects to work with on account of their  time-dependency. We now consider a generational model of the NBP. We can think of each line of descent in the sequence of  neutron creation as genealogies. 
In place of $(X_t, t\geq 0)$, we consider the process $(\mathcal{X}_n, n\geq 0)$, where, for $n\geq 1$, $\mathcal{X}_n$ is $\mathcal{M}(D\times V)$-valued and can be written $\mathcal{X}_n = \sum_{i = 1}^{\mathcal{N}_n}\delta_{(r^{(n)}_i, \upsilon^{(n)}_i )}$, where $\{(r^{(n)}_i, \upsilon^{(n)}_i ), i = 1, \cdots \mathcal{N}_n\}$ are the position-velocity configurations of the $\mathcal{N}_n$ particles
that are $n$-th in their genealogies to be the result of a fission event. $\mathcal{X}_0$ is consistent with $X_0$ and is the initial configuration of neutron positions and velocities. As such, for $n\geq 1$ we can think of $\mathcal{X}_n$ as the $n$-th generation of the system and we refer to them as the neutron generational process (NGP). The reader who is more experienced with the theory of branching processes will know $\mathcal{X}_n$ to be an example of what is called a stopping line; see \cite{Kyp}.

\bigskip

Appealing to the obvious meaning of $\langle g, \mathcal{X}_n\rangle$, define the expectation semigroup $(\Psi_n, n \ge 0)$ by
\begin{equation}
\Psi_n[g](r, \upsilon) = \mathbb{E}_{\delta_{(r, \upsilon)}}\left[\langle g, \mathcal{X}_n\rangle \right], \qquad n\geq 0, r\in D, \upsilon\in V,
\label{expsemi}
\end{equation}
with $\Psi_0[g]: = g\in L_\infty^+(D\times V)$.
The main motivation for introducing the NGP  is that, just as we have seen that the meaning of \eqref{lambda} can be phrased in terms of a multiplicative invariance with respect to the solution of an ACP \eqref{ACP} or of the mild equation \eqref{mild}, we want to identify the eigen-problem \eqref{keval2} in terms of the semigroup above.

\bigskip

To this end, let us introduce the problem of finding a pair $k > 0$ and $\phi\in L^+_\infty(D\times V)$ such that, pointwise, 
\begin{equation}
\Psi_1[\phi](r,\upsilon) = k\phi(r,\upsilon), \qquad r\in D, \upsilon\in V.
\label{prob}
\end{equation}
In the next section we will show the existence of a solution to \eqref{prob} which also plays an important role in the asymptotic behaviour of $\Psi_n$ as $n\to\infty$. Before getting there, let us give a heuristic argument as to why \eqref{prob} is another form of the eigenvalue problem \eqref{keval2}.

\bigskip

By splitting on the first fission event, $\Psi_n$ solves the following mild equation
\begin{equation}
\Psi_n[g](r, \upsilon) = \int_0^\infty \texttt{Q}_s\left[{\mathcal F} \Psi_{n-1}[g]\right](r, \upsilon){\rm d}s,\qquad r\in D, \upsilon\in V, g\in L_\infty^+(D\times V),
\label{mildNTE}
\end{equation}
where $(\texttt{Q}_s,s \ge 0)$ is the expectation semigroup associated with the operator $\mathcal{T} + \mathcal{S}$. More precisely,
\begin{equation}
\texttt{Q}_s[g](r, \upsilon) = \mathbb{E}_{\delta_{(r, \upsilon)}}\left[{\rm e}^{-\int_0^s\sigma_\texttt{f}(R_u, \Upsilon_u){\rm d}u}g(R_s, \Upsilon_s)\mathbf{1}_{(s < \tau_D)}\right],
\label{T+S}
\end{equation}
where $(R_s, \Upsilon_s)_{s \ge 0}$ is the $\sigma_\texttt{s}\pi_\texttt{s}$-NRW.
Then, if the pair $(k, \phi)$ solves \eqref{prob}, the strong Markov property along with an iteration implies that
\[
k^n\phi(r, \upsilon) = \Psi_n[\phi](r, \upsilon)
, \qquad r\in D, \upsilon\in V.
\]
Using it in~\eqref{mildNTE} and dividing through by $k^n$ yields
\begin{equation}
\phi(r, \upsilon) = \int_0^\infty \texttt{Q}_s\left[\frac{1}{k}\mathcal{F}\phi\right](r, \upsilon) {\rm d}s.
\label{mildeval}
\end{equation}
Now set 
\[
V_t \coloneqq \int_0^t \texttt{Q}_s\left[g\right](r, \upsilon){\rm d}s.
\]
Then, heuristically speaking, since $\texttt{Q}$ is associated to the generator $\mathcal{T}+\mathcal{S}$, classical Feynman-Kac theory suggests that $V_t$ `solves' the equation
\[
\frac{\partial V_t}{\partial t} = (\mathcal{T} + \mathcal{S})V_t + g.
\]
with $V_0 = 0$. Note that $\partial V_t/\partial t = \Q_t[g]$, which tends to zero as $t \to \infty$ thanks to the transience of $(R,\Upsilon)$.
Hence, taking $g = {k}^{-1}\mathcal{F}\phi$, letting $t \to \infty$ in the above equation, recalling that $(\texttt{Q}_s,s \ge 0)$ is the expectation semigroup associated with the operator $\mathcal{T} + \mathcal{S}$, and using the identity~\eqref{mildeval} yields
\[
0 = (\mathcal{T} + \mathcal{S})\phi + \frac{1}{k}\mathcal{F}\phi.
\]

\section{Probabilistic  solution to \eqref{keval2}}\label{QSD}

In this section we state our main result regarding the existence of the eigenvalue and eigenfunction as specified by \eqref{prob}. We are once more motivated by the ideas presented in~\cite{CV} and will use some of the techniques that were further developed in~\cite{SNTE}.

\bigskip

We start by constructing the many-to-one formula that is associated to the semigroup $(\Psi_n,n\geq 0)$ in the spirit of the two representations of $(\psi_t, t\geq 0)$ given in Sections \ref{psi1} and \ref{psi2}. In this case it takes a slightly different form to the one in the time-dependent case. For ease of notation, let
\[
m(r, \upsilon) \coloneqq \int_V\pi_\texttt{f}(r, \upsilon, \upsilon'){\rm d}\upsilon',
\]
denote the mean number of offspring generated by a fission event at $(r, \upsilon)$, and let $(T_n, n \ge 1)$ denote the time of the scatter event in the $\alpha\pi$-NRW that corresponds to the $n$-th fission event in the corresponding NBP, $X$. 
\bigskip

More formally, referring to, \eqref{alpha}, we can think of the $\alpha\pi$-NRW at each scatter event as follows. For $k\geq 1$, when  the NRW $(R,\Upsilon)$ scatters for the $k$-th time at $(r,\upsilon)$ (with rate $\alpha(r,\upsilon)$), a coin is tossed and the random variable $\texttt{I}_k(r, \upsilon)$ takes the value $1$ with probability $
{\sigma_{\texttt{f}}(r,\upsilon)m(r,\upsilon)}/{ ({\sigma_{\texttt s}(r,\upsilon) + \sigma_{\texttt f}(r,\upsilon)m(r, \upsilon)}) }
$ and its new velocity, is selected according to an independent copy of the random variable $\Theta^{\texttt f}_k(r, \upsilon)$, whose distribution has  probability density $\pi_{\texttt f}(r,\upsilon, \upsilon')/m(r,\upsilon)$. On the other hand, with probability 
 ${\sigma_{\texttt s}(r,\upsilon)}/({\sigma_{\texttt s}(r,\upsilon) + \sigma_{\texttt f}(r,\upsilon)m(r, \upsilon)})$
the random variable $\texttt{I}_k(r, \upsilon)$  takes the value $0$
and its new velocity, is selected according to an independent copy of the random variable $\Theta^{\texttt s}_k(r, \upsilon)$, whose distribution has probability density $\pi_{\texttt s}(r,\upsilon,\upsilon')$. As such, the  velocity immediately after the $n$-th scatter of the NRW, given that the position-velocity configuration immediately before is $(r,\upsilon)$, is coded by the random variable 
\[
\texttt{I}_k (r,\upsilon)\Theta^{\texttt f}_k(r,\upsilon) + (1-\texttt{I}_k(r,\upsilon) )\Theta^{\texttt s}_k(r,\upsilon).
\]
We thus can identify sequentially, $T_0 = 0$ and, for $n\geq 1$,
\[
T_n = \inf\{t>T_{n-1} : \Upsilon_{t}\neq \Upsilon_{t-} \text{ and }\texttt{I}_{k_t}(R_t, \Upsilon_{t-}) =1\},
\]
where $(k_t, t\geq 0)$ is the process counting the number of scattering events of the NRW up to time $t$.
\bigskip

{\color{black}Note, for the above construction of indicators to make sense, we should at least have some region of space for which fission can take place. As such the assumption (H3) becomes relevant here.} Analogously to Lemma~\ref{timeM21}, we have the following many-to-one formula associated with the NBP. 
\begin{lem} {\color{black} Suppose (H1), (H2) and (H3) hold.} The solution to \eqref{mildNTE} among the class of expectation semigroups is unique for $g\in L_\infty^+(D\times V)$ and 
the semigroup $(\Psi_n, n\geq 0)$ may alternatively be represented\footnote{Here, we define $\prod_{i = 1}^0\cdot \coloneqq 1$.} as 
\begin{equation}
\Psi_n[g](r, \upsilon) = \mathbf{E}_{(r, \upsilon)}\left[\prod_{i = 1}^n m(R_{T_i}, \Upsilon_{T_i-})g(R_{T_n}, \Upsilon_{T_n})\mathbf{1}_{(T_n < \kappa^D)} \right],\qquad r\in D, \upsilon\in V, n\geq 1,
\label{M21}
\end{equation}
(with $\Psi_0[g] =g$), where $(R_t, \Upsilon_t)_{t \ge 0}$ is the $\alpha\pi$-NRW, and 
\[
\kappa^D \coloneqq \inf\{t > 0 : R_{t} \notin D\}.
\]
\end{lem}
\begin{proof}
We first note that the sequence  $(\Psi_n, n\geq 0)$ as defined in \eqref{M21} is a semigroup since, due to the strong Markov property, we have 
\begin{align*}
&\Psi_{n+m}[g](r, \upsilon) \\
&= \mathbf{E}_{(r, \upsilon)}\left[\mathbf{E}\left[\prod_{i = 1}^{n+m} m(R_{T_i}, \Upsilon_{T_i-}) g(R_{T_{n + m}}, \Upsilon_{T_{n + m}})\mathbf{1}_{(T_{n + m} < \kappa^D)} \bigg| \mathcal{F}_n\right]\right] \\
&= \mathbf{E}_{(r, \upsilon)}\left[\prod_{i = 1}^n m(R_{T_i}, \Upsilon_{T_i-}) \mathbf{E}_{(R_{T_n}, \Upsilon_{T_n})}\left[\prod_{i = 1}^m m(R_{T_i}, \Upsilon_{T_i-}) g(R_{T_m}, \Upsilon_{T_m})\mathbf{1}_{(T_m < \kappa^D)} \right]\mathbf{1}_{(T_n < \kappa^D)}\right]\\ 
&= \Psi_n[\Psi_m[g]](r, \upsilon), \qquad r\in D, \upsilon\in V.
\end{align*}

In order to show that $\Psi_n$  as defined in \eqref{M21} does indeed solve~\eqref{mildNTE}, we consider the process at time $T_1$. Before doing so, we first note that the $\alpha\pi$-NRW  is exactly the same as the $\sigma_\texttt{s}\pi_\texttt{s}$-NRW over the  time interval $[0,T_1)$ and, at time $T_1$, the velocity of the former is chosen according to the expectation operator
\[
\tilde{\mathcal{F}}[g](r, \upsilon) \coloneqq \int_Vg(r, \upsilon')\frac{\pi_\texttt{f}(r, \upsilon, \upsilon')}{m(r, \upsilon)}{\rm d}\upsilon'.
\]
Then, applying the strong Markov property at time $T_1$,
\begin{align*}
\Psi_n[g](r, \upsilon) &= \mathbf{E}_{(r, \upsilon)}\left[\prod_{i = 1}^n m(R_{T_i}, \Upsilon_{T_i-})g(R_{T_n}, \Upsilon_{T_n})\mathbf{1}_{(T_n < \kappa^D)} \right]\\
&= \mathbf{E}_{(r, \upsilon)}\left[m(R_{T_1}, \Upsilon_{T_1-})\tilde{\mathcal{F}}[\Psi_{n-1}[g]](R_{T_1}, \Upsilon_{T_1-})\mathbf{1}_{(T_1 < \kappa^D)} \right]\\
&= \int_0^\infty \mathbf{E}_{(r, \upsilon)}\left[ \sigma_\texttt{f}(R_s, \Upsilon_s){\rm e}^{-\int_0^s{\sigma_\texttt{f}(R_u, \Upsilon_u){\rm d}u}}m(R_s, \Upsilon_{s-})\tilde{\mathcal{F}}[\Psi_{n-1}[g]](R_s, \Upsilon_{s-})\mathbf{1}_{(s < \kappa^D)} \right]{\rm d}s\\
&= \int_0^\infty \texttt{Q}_s[\mathcal{F}\Psi_{n-1}[g]](r, \upsilon){\rm d}s,
\end{align*}
where the final equality follows from the fact that $m\sigma_\texttt{f}\tilde{\mathcal{F}} = \mathcal{F}$.

\bigskip

It remains to show that \eqref{mildNTE}  has a unique solution for $g\in L_\infty^+(D\times V)$ among the class of expectation semigroups, suppose that $(\Psi_n', n\geq 0)$ is another such solution with $\Psi_0'=g\in L^+_\infty(D\times V)$. Define $\Phi_n = \Psi_n -\Psi_n'$, for $n\geq 0$, and note by linearity that $(\Phi_n , n\geq 0)$ is another expectation semigroup with $\Phi_0 = 0$. Moreover, by linearity $(\Phi_n, n\geq 0)$ also solves \eqref{mildNTE}. On account of this, it is  straightforward to see by induction that if $\Phi_n = 0$ then $\Phi_{n+1} = 0$.  The uniqueness of \eqref{mildNTE} in the class of expectation semigroups thus follows.
\end{proof}

The next result will provide the existence of a solution to \eqref{prob} by working directly with a variant of the semigroup $(\Psi_n, n\geq 0)$. To this end, note that, under the assumption (H4), for non-negative functions $g$ that are bounded by one, say, we have
\begin{equation}
\mathbb{E}_{\delta_{(r, \upsilon)}}\left[\langle g, \mathcal{X}_1\rangle \right] \le \Vert g \Vert_\infty \mathbb{E}_{\delta_{(r, \upsilon)}}\left[\langle 1, \mathcal{X}_1\rangle \right] \le N_\texttt{max}.
\end{equation}
Dividing both sides of the above inequality yields a sub-Markovian  semigroup. Indeed,
\begin{align}
\Psi^\dagger_n[g](r, \upsilon) &\coloneqq N_\texttt{max}^{-n}\Psi_n[g](r, \upsilon) \notag\\
&= \mathbf{E}_{(r, \upsilon)}\left[\prod_{i = 1}^n \frac{m(R_{T_i}, \Upsilon_{T_i-})}{N_\texttt{max}}g(R_{T_n}, \Upsilon_{T_n})\mathbf{1}_{(T_n < \kappa^D)} \right]\notag\\
&= \mathbf{E}_{(r, \upsilon)}\left[g(R_{T_n}, \Upsilon_{T_n})\mathbf{1}_{(T_n < \kappa^D, \, n< \Gamma)} \right]\notag\\
&\eqqcolon \mathbf{E}_{(r, \upsilon)}^\dagger\left[g(R_{T_n}, \Upsilon_{T_n}) \right],
\label{daggerdef}
\end{align}
where $ \Gamma = \min\{n\geq 0: \texttt{K}_n(R_{T_n}, \Upsilon_{T_n -}) =1\}$ where, for $n\geq 0$, $r\in D$ and $\upsilon \in V$, the random variable  $\texttt{K}_{n}(r,\upsilon)$ is an  independent indicator random variable which is equal to 0 with probability ${m(r, \upsilon)}/{N_{\texttt{max}}}$ (note, from the assumptions in Section \ref{intro}, it follows that $ \sup_{r\in D, \upsilon\in V} m(r,\upsilon)\leq N_{\texttt{max}}$).

\bigskip

We are now ready to state the main result of this section, and indeed the article. As its proof is quite lengthy we will delay it until Section \ref{longproof}.  We will need the following stronger assumption of (H3):
\bigskip

{\color{black}
{\bf 
(H3)$^*$: The fission cross section satisfies $\textstyle{\inf_{r \in D, \upsilon, \upsilon' \in V}\sigma_\texttt{f}(r, \upsilon)\pi_\texttt{f}(r, \upsilon, \upsilon') > 0}$.
}
}

\begin{theorem}\label{CVtheorem}
{\color{black}Under the assumptions (H1), (H3)$^*$ and (H4),} for the semigroup  $(\Psi_n, n\ge 0)$  identified by \eqref{mildNTE}, there exist $k_*\in\mathbb{R}$, a positive\footnote{To be precise, by a positive eigenfunction, we mean a mapping from $D\times V\to (0,\infty)$. This does not prevent it being valued zero on $\partial D$, as $D$ is open.} right eigenfunction $\varphi \in L^+_\infty(D\times V)$ and a left eigenmeasure, $\eta$, on $D\times V$, both having associated eigenvalue $k_*^n$. Moreover, $k_*$ is the leading eigenvalue in the sense that, for all $g\in L^+_{\infty}(D\times V)$, 
\begin{equation}
\langle\eta, \Psi_n[g]\rangle = k_*^n\langle\eta, g\rangle\quad  \text{(resp. } 
\Psi_n[\varphi] = k_*^n\varphi
\text{)} \quad n\ge 0,
\label{leftandright}
\end{equation}
and there exists $\gamma > 1$ such that, for all $g\in L^+_\infty(D\times V)$,
\begin{equation}
  \sup_{g\in L_\infty^+(D\times V): \norm{g}_\infty \leq1}\left\| k_*^{-n}{\varphi}^{-1}{\Psi_n[g]}-\langle\eta, g\rangle\right\|_\infty = O(\gamma^{- n}) \text{ as } n\rightarrow+\infty.
\label{spectralexpsgp}
\end{equation}
\end{theorem}

\section{Classical existence of solution to \eqref{keval2}}\label{Sclassical}

Our objective here is to make rigorous the sense in which solving \eqref{prob} is consistent with solving the eigenvalue problem \eqref{keval2} in the classical sense.

\bigskip

We begin by considering the abstract Cauchy problem (ACP) on $L_2(D\times V)$,
\begin{equation}
\left\{
\begin{array}{rl}
\dfrac{\partial}{\partial t}u_t &= ({\mathcal T} + {\mathcal S})u_t
\\
u_0& = g.
\end{array}
\right.
\label{ACPgen}
\end{equation}

Then, just as in the spirit of Theorems \ref{thACP} and \ref{thm:co-semi}, it is not difficult to show that the operator $(\mathcal{T} + \mathcal{S}, {\rm Dom}(\mathcal{T} + \mathcal{S}))$ generates a unique solution to \eqref{ACPgen} via the $c_0$-semigroup  $(\mathcal{V}_t, t \ge 0)$ given by
\[
\mathcal{V}_t[g] \coloneqq {\rm exp}(t(\mathcal{T} + \mathcal{S}))g, 
\]
on $L_2(D\times V)$ (and hence for $g \in L_2(D\times V)$).
Moreover, we have that the expectation semigroup   $(\texttt{Q}_t,t \ge 0)$ agrees with $ ( \mathcal{V}_t, t\geq 0)$ on $L_2(D\times V)$, providing $g\in L^+_\infty(D\times V)$. This latter claim follows the same idea as the proof of Theorem 8.1 in \cite{MultiNTE}.

\bigskip

The equivalence of the semigroups $(\texttt{Q}_t,t \ge 0)$ and $ ( \mathcal{V}_t, t\geq 0)$ is what we will use to identify a classical (in the $L_2$-sense)  and probabilistic meaning to \eqref{keval2}. We start by showing the classical existence of a solution to  \eqref{keval2} on $L_2(D\times V)$. We note that this problem has been previously considered in~\cite{Mika, M-K}. In~\cite{Mika}, the author converted the criticality problem~\eqref{keval2} into a time-dependent problem in order to exploit the existing theory for time-dependent problems, whereas the methods used in~\cite[Section 5.11]{M-K} are similar to those presented in~\cite{MultiNTE}.
Another more restrictive version of assumption (H2) is needed, which also implies that (H3) holds:

\bigskip
{\color{black}
{\bf 
(H5): We have $\sigma_{\emph{\texttt{s}}}(r, \upsilon)\pi_{\emph{\texttt{s}}}(r, \upsilon, \upsilon') > 0$ and $\sigma_{\emph{\texttt{f}}}(r, \upsilon)\pi_{\emph{\texttt{f}}}(r, \upsilon, \upsilon') > 0$  on $D\times V\times V$.
} 
}

\begin{theorem}\label{c0eigen} 
Suppose that the cross sections $\sigma_{\emph{\texttt{f}}}\pi_{\emph{\texttt{f}}}$ and $\sigma_{\emph{\texttt{s}}}\pi_{\emph{\texttt{s}}}$ are piecewise continuous\footnote{A function is piecewise continuous if its domain can be divided into an exhaustive finite partition (e.g. polytopes) such that there is continuity in each element of the partition. This is precisely how cross sections are stored in numerical libraries for modelling of nuclear reactor cores.}. {\color{black}Further, assume that (H1) and  (H5) hold.}
Then there exist a real eigenvalue $k > 0$ and associated eigenfunction $\phi \in L_2^+(D\times V)$ such that~\eqref{keval2} holds on $L_2(D\times V)$. Moreover, $k$ can be explicitly identified as
\begin{equation}
k = \sup\left\{|\lambda | : (\mathcal{T} + \mathcal{S})\phi + \frac{1}{\lambda}\mathcal{F}\phi = 0 \text{ for some } \phi \in L_2(D \times V)\right\}.
\label{largest}
\end{equation}
\end{theorem}
\begin{proof} 
We start by considering  a related eigenvalue problem. 
First recall from~\cite{MultiNTE} that, due to the transience of $\mathcal{T}$ on $D$, there exist constants $M_1, \omega > 0$ such that $\Vert {\rm e}^{t \mathcal{T}}\Vert \le M_1{\rm e}^{-\omega t}$ for each $t \ge 0$. Further, since $\mathcal{S}$ is conservative, there exists $M_2 > 0$ such that\footnote{We use the standard definition of operator norm, namely $\norm{A} =\sup_{\norm{f}_2\leq 1}
\norm{\mathcal{A}f}_2$, where, as before, $\norm{\cdot}_2$ is the usual norm on $L_2(D\times V)$.} 
$\Vert {\rm e}^{t \mathcal{S}}\Vert \le M_2$, $t \ge 0$. Hence $\Vert \mathcal{V}_t\Vert \le M{\rm e}^{-\omega t}, t \ge 0$, where $M = M_1M_2$. Then, classical semigroup theory~\cite{OpTh} gives the existence of the resolvent operator $(\lambda \mathcal{I} - (\mathcal{T} + \mathcal{S}))^{-1}$ for all $\lambda$ such that ${\rm Re}\lambda > -\omega$, where $\mathcal{I}$ is the identity operator on $L_2(D \times V)$. In particular, the resolvent is well-defined for $\lambda = 0$. Hence, the eigenvalue problem~\eqref{keval2} is equivalent to
\begin{equation}
-(\mathcal{T} + \mathcal{S})^{-1}\mathcal{F}\phi = k\phi.
\label{keval3}
\end{equation}
Due to the assumptions (H1) and (H5), almost identical arguments to those given in the proof of~\cite[Proposition 9.1]{MultiNTE} show that $-(\mathcal{T} + \mathcal{S})^{-1}\mathcal{F}$ is a positive, compact, irreducible operator. Concluding in the same way as the aforementioned proposition, de Pagter's theorem~\cite[Theorem 5.7]{M-K} implies that its spectral radius, $r(-(\mathcal{T} + \mathcal{S})^{-1}\mathcal{F})$, is strictly positive. It follows from the Krein-Rutman Theorem~\cite[Theorem 9.1]{MultiNTE} that $k \coloneqq r(-(\mathcal{T} + \mathcal{S})^{-1}\mathcal{F}) \coloneqq \sup\{|\lambda| : -(\mathcal{T} + \mathcal{S})^{-1}\mathcal{F}\phi = \lambda\phi \text{ for some } \phi \in L_2(D \times V)\}$ is the leading eigenvalue of the operator $-(\mathcal{T} + \mathcal{S})^{-1}\mathcal{F}$ with corresponding positive eigenfunction $\phi$.
\end{proof}

In a similar manner to~\cite{MultiNTE}, we are able to provide more information about the structure of the spectrum of the operator $-(\mathcal{T} + \mathcal{S})^{-1}\mathcal{F}$.

\begin{prop}\label{spectrum}
Under the assumptions of Theorem~\ref{c0eigen}, the part of the spectrum given by $\sigma(-(\mathcal{T} + \mathcal{S})^{-1}\mathcal{F}) \cap \{\lambda : {\rm Re}(\lambda) > 0\}$ consists of finitely many eigenvalues with finite multiplicities. In particular, $k$ is both algebraically and geometrically simple\footnote{An eigenvalue $\lambda$ associated with an operator $A$ is geometrically simple if dim$({\rm ker}(\lambda I - A)) = 1$ and algebraically simple if $\sup_{k\ge 1}{\rm dim}({\rm ker}(\lambda I - A)^k) = 1$}.
\end{prop}
\begin{proof}
We follow the idea of the proof of~\cite[Theorem 4.13]{M-K} and consider the invertibility of the operator $\lambda\mathcal{I} + (\mathcal{T} + \mathcal{S})^{-1}\mathcal{F}$ by considering the following problem,
\[
\left(\mathcal{I} + \frac{1}{\lambda}(\mathcal{T} + \mathcal{S})^{-1}\mathcal{F}\right)f = \frac{1}{\lambda}g,
\]
for $\lambda \in \sigma(-(\mathcal{T} + \mathcal{S})^{-1}\mathcal{F}) \cap \{\lambda : {\rm Re}(\lambda) > 0\}$. Note that this latter set is non-empty on account of the previous theorem.

\bigskip

As stated in the proof of Theorem~\ref{c0eigen}, the operator $-{\lambda}^{-1}(\mathcal{T} + \mathcal{S})^{-1}\mathcal{F}$ is compact in $L_2(D \times V)$ so that by Gohberg-Shmulyan's Theorem~\cite{RV}, $\left(\mathcal{I} + {\lambda}^{-1}(\mathcal{T} + \mathcal{S})^{-1}\mathcal{F}\right)^{-1}$ exists except for a finite set of discrete degenerate poles. This implies that $\left(\lambda\mathcal{I} + (\mathcal{T} + \mathcal{S})^{-1}\mathcal{F}\right)^{-1}, \lambda \in \sigma(-(\mathcal{T} + \mathcal{S})^{-1}\mathcal{F}) \cap \{\lambda : {\rm Re}(\lambda) > 0\}$ exists except for a finite set of eigenvalues with finite multiplicities. 

\medskip 

We now prove that $k$ is a simple eigenvalue of the operator $-(\mathcal{T} + \mathcal{S})^{-1}\mathcal{F}$. In order to do so, we need to consider the adjoint eigenvalue problem
\begin{equation}
\mathcal{F}^\top({\mathcal T}^\top + \mathcal{S}^\top)^{-1}\phi^\top = k^\top\phi^\top,
\label{adjointeval}
\end{equation}
where ${\mathcal T}^\top$ denotes the adjoint of ${\mathcal T}$, with similar definitions for $\mathcal{F}^\top$ and $\mathcal{S}^\top$.

\bigskip

We first note that, since the operator $\mathcal{T}^\top + \mathcal{S}^\top$ enjoys similar properties to $\mathcal{T} + \mathcal{S}$, the same methods as those given in the proof of Theorem~\ref{L2soln} apply to give existence of a leading eigenvalue $k^\top$ and corresponding eigenfunction $\phi^\top$. Now, due to~\cite[p.184]{Kato}, if $\lambda$ is an isolated eigenvalue of $-(\mathcal{T} + \mathcal{S})^{-1}\mathcal{F}$, then its complex conjugate, $\bar\lambda$, is an isolated eigenvalue of the adjoint of $-(\mathcal{T} + \mathcal{S})^{-1}\mathcal{F}$ with the same multiplicity. Equivalently, for each isolated $\lambda$ solving~\eqref{keval2} with eigenfunction $\phi$, $\bar\lambda$ solves~\eqref{adjointeval} with a corresponding eigenfunction $\phi^\top$ and has the same multiplicity as $\lambda$. In particular, since $k$ is real, it follows that the leading eigenvalue associated with~\eqref{adjointeval} is also $k$. These observations along with similar arguments to those presented in \cite[Theorem 7(iii)]{DL6} and~\cite{Vidav} yield geometric simplicity of $k$. Then straightforward adaptations of the arguments in~\cite[Remark 12]{DL6} yield algebraic simplicity.
\end{proof}

The next result shows that if we can find a solution to \eqref{keval2}, then it must necessarily agree with the eigensolution constructed in Theorem \ref{CVtheorem} on $L_2(D\times V)$.

\begin{theorem}\label{L2soln}Suppose the assumptions of Theorem~\ref{c0eigen} are in force\footnote{Note that these assumptions imply those required for Theorem~\ref{CVtheorem}.}, that  $(k_*, \phi_*)$ solves \eqref{prob} and $(k, \phi)$ denotes the leading eigensolution  to \eqref{keval2}. Then $k = k_*$, and, up to a positive constant multiple, $\phi$ agrees with $\phi_*$ on $L_2(D\times V)$. 
\end{theorem}
\begin{proof}
Recall the semigroup, $(\mathcal{V}_t)_{t \ge 0}$, generated by $\mathcal{T} + \mathcal{S}$ and note that, due to boundedness of the operator $\mathcal{F}$, if $g \in L_p(D\times V)$, then  $\mathcal{F}g \in L_p(D\times V)$, $p \in [1, \infty]$. Thanks to~\cite[Chapter II, Lemma 1.3]{EN}, $(\mathcal{V}_t)_{t \ge 0}$ satisfies
\begin{equation}
\mathcal{V}_t[\mathcal{F}g] = (\mathcal{T} + \mathcal{S})\int_0^t\mathcal{V}_s[\mathcal{F}g]{\rm d}s + \mathcal{F}g.
\label{eq1}
\end{equation}
Letting $t \to \infty$ in the above equation, we obtain
\begin{equation}
0 = (\mathcal{T} + \mathcal{S})\int_0^\infty\mathcal{V}_s[\mathcal{F}g]{\rm d}s + \mathcal{F}g,
\label{eq2}
\end{equation}
which follows from the fact that $(\mathcal{T} + \mathcal{S})$ is a transient operator so that $\lim_{t \to\infty}\mathcal{V}_t[g] = 0$. Setting $g = \phi_*$ in~\eqref{eq2} and using the fact that $(\texttt{Q}_s, s\geq 0)$ and $(\mathcal{V}_s, s\geq 0)$ agree on $L_2(D\times V)$,  providing $g\in L^+_\infty(D\times V)$, yields
\begin{equation}
0 = (\mathcal{T} + \mathcal{S})\int_0^\infty\texttt{Q}_s[\mathcal{F}\phi_*]{\rm d}s + \mathcal{F}\phi_*.
\label{eq3}
\end{equation}
Now taking advantage of \eqref{prob} for $\phi_*$, noting in particular \eqref{mildNTE}, we have
\begin{equation}
\int_0^\infty\texttt{Q}_s[\mathcal{F}\phi_*] = \Psi_1[\phi_*] = k_*\phi_*.
\label{eq4}
\end{equation}
Substituting this into~\eqref{eq3} shows that $(k_*, \phi_*)$ is a solution to~\eqref{keval2} on $L_2(D\times V)$.

\bigskip

To conclude the proof, we first show that $k_* = k$. Again, consider the adjoint problem~\eqref{adjointeval} and note that
\begin{align*}
0 &= \langle (\mathcal{T} + \mathcal{S})^{-1}\mathcal{F}\phi_*, \phi^\top\rangle -  \langle \mathcal{F}^\top(\mathcal{T}^\top + \mathcal{S}^\top)^{-1}\phi^\top, \phi_*\rangle \\
&= (k-k_*)\langle \phi^\top, \phi_*\rangle.
\end{align*}
Since $\phi_*$ and $\phi^\top$ are positive, we must have $k_* = k$. Due to simplicity of $k$ from the previous proposition, it follows that $\phi = \phi_*$ up to a multiplicative constant.
\end{proof}

\section{Proof of Theorem~\ref{CVtheorem}}\label{longproof}
As previously stated, our methods of proving Theorem~\ref{CVtheorem} are motivated by those used in~\cite{SNTE, CV}. The main part of the proof comes from~\cite[Theorem 2.1]{CV}, which we restate (in the language of the desired application) here for convenience. To this end, recalling the notation in \eqref{daggerdef}, define 
\[
\texttt{k} = \Gamma\wedge \min\{n\geq 1: T_n\geq \kappa^D\}.
\]
\begin{theorem}\label{7CVtheoremBis} {\color{black} Suppose that (H1), (H3)$^*$ and (H4) are in force.}
Suppose that there exists a probability measure $\nu$ on $D\times V$ such that
\begin{enumerate}
\item[\namedlabel{itm:A1}{(A1)}] there exist $n_0$, $c_1 > 0$ such that for each $(r, \upsilon)\in D\times V$,
\[
\mathbf{P}_{(r, \upsilon)}((R_{T_{n_0}}, \Upsilon_{T_{n_0}}) \in \cdot \;| n_0 < \texttt{k}) \ge c_1 \nu(\cdot);
\]
\item[\namedlabel{itm:A2}{(A2)}] there exists a constant $c_2 > 0$ such that for each $(r, \upsilon)\in D\times V$ and for every $n\ge 0$,
\[
\mathbf{P}_{\nu}(n < \texttt{k}) \ge c_2\mathbf{P}_{(r, \upsilon)}(n <  \texttt{k}).
\]
\end{enumerate}
Then, there exists $k_c \in(0, 1)$ such that, there exist an
eigenmeasure $\eta$ on $D\times V$ and a positive right eigenfunction
$\varphi$ of $\Psi_n^\dagger$ (defined in~\eqref{daggerdef})
with eigenvalue
$k_c^n$, such that $\eta$ is a probability
measure and $\varphi\in L^+_\infty(D\times V) $, i.e. for all
$g\in L_{\infty}(D\times V)$,
\begin{equation}
\eta [\Psi^{\dagger}_n[g] ]= k_c^n\eta[g]\quad  \text{and}\quad 
\Psi^{\dagger}_n[\varphi] = k_c^n\varphi
 \quad n\ge 0.
 \label{eta}
\end{equation}
Moreover, there exist $C,\gamma>0$ such that, for $g\in L^+_\infty(D\times V)$ and $n$ sufficiently large (independently of $g$),
\begin{equation}
\left\| k_c^{-n}\varphi^{-1}\Psi_n^{\dagger}[g]-\eta[g]\right\|_\infty\leq C\gamma^{- n}\|g\|_\infty.
\label{7spectralexpsgp}
\end{equation}
In particular, setting $g \equiv 1$, as $n\to\infty$,
\begin{equation}
\left\| k_c^{-n} \varphi^{-1}\mathbf{P}_{\cdot}(n < \texttt{k}) - 1\right\|_\infty\leq C\gamma^{- n}.
\label{7die}
\end{equation}
\end{theorem}

\bigskip
It is then straightforward to conclude that $\eta$ and $\varphi$ are the left eigenmeasure and right eigenfunction corresponding to the eigenvalue $k_* = k_cN_{\texttt{max}}$ for the semigroup $\Psi_n$. 

\bigskip
%
%
%

We now proceed to the proof of Theorem~\ref{7CVtheoremBis}. We will use the notation $J_k$ to denote the $k^{th}$ scatter event of the random walk $(R, \Upsilon)$ under ${\bf P}^\dagger$ and {\color{black} recall that  $T_k$  denotes the scatter event that corresponds to the $k^{th}$ fission event in the original NBP.  The basis of our proof relies on the fact that, for each $k\geq 1$, $T_k = J_k$ with positive probability. }

\bigskip

A fundamental part of the proof of~\ref{itm:A1} and~\ref{itm:A2} is the following lemma. We refer the reader to~\cite[Lemma 7.3]{SNTE} for its proof.

\begin{lem}\label{jumps} {\color{black} Under the assumptions of Theorem \ref{7CVtheoremBis}, }
for all $r\in D$ and $\upsilon \in V$, we have
\begin{equation}
\mathbf{P}_{(r, \upsilon)}^\dagger(J_7 < \emph{\texttt{k}}, R_{J_7} \in {\rm d}z) \le C\mathbf{1}_{(z\in D)}\,{\rm d}z,
\label{jump7}
\end{equation}
for some constant $C > 0$, and
\begin{equation}
\mathbf{P}^\dagger_{\nu}(J_1 <\emph{\texttt{k}}, R_{J_1} \in {\rm d}z) \ge c\mathbf{1}_{(z\in D)}\,{\rm d}z,
\label{jump1}
\end{equation}
for another constant $c>0$, where $\nu$ is Lebesgue measure on $D\times V$.
\end{lem}

\begin{proof}[Proof of~\ref{itm:A1}]
In order to prove~\ref{itm:A1}, we use similar arguments to those presented in the proof of~\eqref{jump1}. To this end, fix $r_0 \in D$ and suppose $\Upsilon_0$ is uniformly distributed on $V$. Then, due to the assumptions {\color{black} (H1) and (H3)$^*$}, the techniques used in~\cite{SNTE} to prove~\eqref{jump1} yield
\begin{equation}
\mathbf{E}_{(r_0, \Upsilon_0)}\left[ f(R_{J_1})\mathbf{1}_{(T_1 = J_1)}\right] \ge C_0 \int_D {\rm d}z \mathbf{1}_{([r, z] \subset D)} f(z).
\label{LB1}
\end{equation}
Recall the (deterministic) quantity $\kappa_{r_0, \upsilon_0}^D$, for $r_0 \in D$, $\upsilon_0 \in V$, defined in Theorem \ref{mild}. Also note that due to (H3)$^*$, $\pi$ is bounded below by a constant (see discussion just before Lemma 7.2 of~\cite{SNTE}). Using this, along with the strong Markov property, (H1) and \eqref{LB1}, we have
\begin{align}
\mathbb{E}_{(r_0, \upsilon_0)}^\dagger[f(R_{T_2}, \Upsilon_{T_2})\mathds{1}_{(T_2 = J_2)}] & \ge C_1\int_0^{\kappa_{r_0, \upsilon_0}^D}{\rm d}s{\rm e}^{-\bar\alpha s}\underline{\pi}\int_V {\rm d}\upsilon_1\mathbf{E}_{(r_0 + \upsilon_0 s, \upsilon_1)}^\dagger[f(R_{J_1}, \Upsilon_{J_1})\mathds{1}_{(T_1 = J_1)}]\notag\\
& \ge C_2 \kappa_{r_0, \upsilon_0}^D\int_D{\rm d}r\int_V{\rm d}\upsilon f(r, \upsilon) \label{LB2}.
\end{align}
Finally, we note that due to (H1) and (H3)$^*$,
\begin{equation}
\mathbf{P}_{(r_0, \upsilon_0)}^\dagger(T_2 < \texttt{k}) \le \mathbf{P}^\dagger(J_1 < \texttt{k}) \le \int_0^{\kappa_{r_0, \upsilon_0}^D}{\rm d}s\bar\alpha{\rm e}^{-\underline{\alpha}s} \le C_3 \kappa_{r_0, \upsilon_0}^D.
\label{LB3}
\end{equation}
Combining this with \eqref{LB2} yields (A1) with $\nu$ as Lebesgue measure on $D \times V$ and $n_0 = 2$.
\end{proof}

We now prove~\ref{itm:A2}. Again, we use a similar method to the one used in~\cite{SNTE}, however, we state the proof in full to illustrate where the differences occur. 

\begin{proof}[Proof of A2]
Let $n \ge 7$ and note that $T_n - J_7  \ge T_n - T_7 $. This and the strong Markov property imply
\begin{align}
\mathbf{P}_{(r, \upsilon)}(n < \texttt{k}) &\le \mathbf{E}_{(r, \upsilon)}^\dagger\left[ \mathbf{P}_{(R_{J_7}, \Upsilon_{J_7})}\left(n - 7 < \texttt{k} \right)\right]\notag \\
&\le C'\int_D\int_V \mathbf{P}_{(z, w)}\left(n - 7 < \texttt{k}  \right){\rm d}z{\rm d}w, \label{ineq1}
\end{align}
where we have used Lemma~\ref{jumps} to obtain the final inequality.

\bigskip

Now suppose $n \ge 1$. Recalling the measure $\nu$ from~\ref{itm:A1}, another application of Lemma~\ref{jumps} gives
\begin{align}
\mathbf{P}_\nu(n < \texttt{k}) &= \mathbf{E}_\nu^\dagger\left[\mathbf{1}_{(J_1 < \texttt{k})}\mathbf{P}_{(R_{J_1}, \Upsilon_{J_1})}(n < \texttt{k}) \right]\notag\\
&\ge c'\int_D\int_V\mathbf{P}_{(z, w)}(n < \texttt{k}){\rm d}z{\rm d}w. \label{ineq2}
\end{align}

\bigskip

Then, for $n \ge 8$, combining ~\eqref{ineq1} and~\eqref{ineq2} yields
\begin{equation}
\mathbf{P}_{(r, \upsilon)}(n < \texttt{k}) \le \frac{C'}{c'}\mathbf{P}_\nu\left(n - 7 < \texttt{k}\right).
\label{ineq3}
\end{equation}
Now recalling $n_0$ from~\ref{itm:A1}, it follows from~\ref{itm:A1} that
\begin{equation}
\mathbf{P}_\nu^\dagger((R_{T_{n_0}}, \Upsilon_{T_{n_0}}) \in \cdot) \ge c_1\mathbf{P}_\nu(n_0 < \texttt{k} )\nu(\cdot).
\label{ineq4}
\end{equation}
Again, due to assumptions (H1) and (H3)$^*$,
\begin{equation}
\mathbf{P}_\nu(n_0 < \texttt{k}) \ge \int_{D \times V}{\bf P}_{(r, \upsilon)}^\dagger(T_{n_0} = J_{n_0}, n_0 < \texttt{k}) \nu({\rm d}r, {\rm d}\upsilon) \ge K,
\label{killprob}
\end{equation}
for some constant $K > 0$. Then, for $n \ge 8$, due to~\eqref{ineq4} and~\eqref{killprob},
\begin{align}
\mathbf{P}_\nu\left(n - 7 + n_0  < \texttt{k}\right) &= \mathbf{E}_\nu\left[\mathbf{1}_{(n_0 < \texttt{k})}\mathbf{P}_{(R_{T_{n_0}}, \Upsilon_{T_{n_0}})}\left( n - 7 < \texttt{k}\right) \right]\notag \\
& \ge Kc_1\mathbf{P}_\nu\left(n - 7 < \texttt{k}  \right). \label{ineq5}
\end{align}
Finally, noting that for $n \ge 1$ we have $n - 7 + 4n_0 \ge n$, so that
\[
\mathbf{P}_\nu(n < \texttt{k}) \ge \mathbf{P}_\nu\left(n - 7 + 4n_0 < \texttt{k}\right),
\]
and applying~\eqref{ineq5} four times implies
\begin{equation}
\mathbf{P}_\nu(n < \texttt{k}) \ge (Kc_1)^4\mathbf{P}_\nu\left(n - 7 < \texttt{k} \right).
\end{equation}
Combining this with~\eqref{ineq3} yields the result.
\end{proof}
\section{Concluding remarks}\label{discussion}

We complete this paper with a number of remarks that reflect on the main theorem here and previous work in \cite{MultiNTE, SNTE, SNTE-II, MCNTE}.
\subsection{$\lambda$-, $k$- and $c$-eigenvalue problems}
There is a third eigenvalue problem associated with the NTE: find $(c, \varphi_c)$ such that
\begin{align*}
\mathcal{T}\varphi_c + \frac{1}{c}\left(\mathcal{S}+ \mathcal{F}\right)\varphi_c = 0.
\end{align*}
The associated mild form of this eigenvalue problem is
\begin{equation}
\texttt{S}_t[\varphi_c](r, \upsilon) + \frac{1}{c}\int_0^t \texttt{S}_s[(\mathcal{S} + \mathcal{F})\varphi_c](r, \upsilon){\rm d}s = \varphi_c (r, \upsilon),
\label{cmild}
\end{equation}
where 
\[
\texttt{S}_t[g](r, \upsilon) = {\rm e}^{-\int_0^t\sigma(r + \upsilon s, \upsilon){\rm d}s}g(r + \upsilon t, \upsilon)\mathbf{1}_{( t <\kappa_{r,\upsilon}^D)}.
\]
By considering the semigroup $\Pi_n[g](r, \upsilon) = \mathbb{E}_{\delta_{(r, \upsilon)}}[\langle g, \mathbb{X}_n\rangle]$, where $\mathbb{X}_n$ is the neutron population at the $n^{th}$ collision (either a scatter or a fission), almost identical proofs to those given in the previous sections yield the existence of the $(c, \varphi_c)$, both in the classical sense and the probabilistic one.

\bigskip
 
In this case, the eigenvalue $c$ can be interpreted as the ratio between neutron production (from both scattering and fission) and neutron loss (due to absorption and leakage). Alternatively, it can be seen as the number of secondary neutrons per collision, rather than only collisions due to fission events. 

\subsection{Martingale convergence and strong law of large numbers}
In a similar fashion to~\cite{SNTE}, Theorem~\ref{CVtheorem} implies that
\[
\mathcal{W}_n \coloneqq k^{-n}\frac{\langle \varphi, \mathcal{X}_n\rangle}{\langle \varphi, \mu\rangle},
\]
is a non-negative mean one martingale under $\mathbb{P}_{\delta_{(r, \upsilon)}}$. One could then show that $(\mathcal{W}_n)_{n \ge 0}$ converges in $L^2(\mathbb{P})$ in the supercritical case, and otherwise has a degenerate limit.

\bigskip

One could also reconstruct the arguments presented in~\cite{SNTE-II} to characterise the growth in the supercritical regime to obtain a strong law of large numbers:
\[
\lim_{n\to\infty} k^{-n}\frac{\langle g, \mathcal{X}_n\rangle}{\langle\varphi, \mu\rangle}=  \langle g,\tilde{\varphi}\rangle \mathcal{W}_\infty,
\]
where $\mathcal{W}_\infty$ is the limit of the martingale $(\mathcal{W}_n)_{n \ge 0}$.

\bigskip

We leave these arguments as an exercise to the reader to avoid unnecessary repetition.

\subsection{Monte-Carlo considerations}
We end this paper with a discussion of the existing Monte Carlo methods for calculating $k_{\texttt{eff}}$ and the associated eigenfunctions, and how we may use the semigroup approach to propose comparable algorithms, similar in style to those presented in~\cite{MCNTE}.

\bigskip

Due to the interpretation of the eigenvalue $k_{\texttt{eff}}$, most of the existing methods in the numerical analysis and engineering literature are based on iterative methods. For example, several algorithms are given in~\cite{SG} that demonstrate how to calculate $k_{\texttt{eff}}$ and $\varphi$. The main idea is to start with a set of $N$ neutrons, distributed in $D\times V$ according to some function  $\varphi^{(0)}$ that serves as an initial guess\footnote{In practice, this is usually either the uniform distribution, or the solution to a diffusion approximation of the eigenvalue problem.} at $\varphi$. The system of neutrons then evolves until the first generation of fission events. Letting $\hat{\varphi}^{(1)}$ be the distribution of these first generation neutrons, the first approximation, $\varphi^{(1)}$, of the eigenfunction $\varphi$ is then obtained by normalising\footnote{This is usually done by either setting $\varphi^{(1)} = \widetilde{\varphi}^{(1)}/\Vert \widetilde{\varphi}^{(1)}\Vert$ or by sampling $N$ neutrons according to $\widetilde{\varphi}^{(1)}$} $\hat{\varphi}^{(1)}$. At the same time, the eigenvalue $k_{\texttt{eff}}$ is approximated by 
\[
k^{(1)} = \frac{\langle\mathbf{1}, \mathcal{F}\varphi^{(1)}\rangle}{\langle \mathbf{1}, (\mathcal{T} + \mathcal{S})\varphi^{(1)}\rangle},
\]
which corresponds to the ratio of source neutrons for generation $2$ to the number of paths simulated in generation $1$. The process is then repeated using $\varphi^{(1)}$ as the initial distribution of neutrons, in order to obtain $\varphi^{(2)}$ and $k^{(2)}$, and so on.


However, some of the methods presented in the literature lead to bias and correlations between the neutrons in successive fission generations. To overcome this problem, the notion of superhistory powering was introduced in~\cite{SHP}. This idea is based on letting the initial set of neutrons evolve for some number, $L$, of generations until the estimates for $k_{\texttt{eff}}$ and $\varphi$ are computed. It is usual in industry to set $L = 10$.

\bigskip

As we have shown in the previous sections, solving~\eqref{keval2} is equivalent to look for the leading eigentriple $(k_*, \varphi, \tilde\varphi)$ of the semigroup $\Psi_n$. Heuristically, this pertains to finding functions $\varphi$ and $\tilde\varphi$ that describe where neutron production (due to fission events) is most prominent, and a parameter $k_*$ that describes the average growth of the number of neutrons in the system. We may use the asymptotics~\eqref{spectralexpsgp} to inform Monte Carlo methods for the calculation of $k_*$, $\varphi$ and $\tilde\varphi$. Indeed, we have
\[
k_* = \lim_{n \to \infty}\frac{1}{n}\log\Psi_n[\mathbf{1}](r, \upsilon),
\]
where $\mathbf{1}$ is the constant function with value one. Here, as an expectation, $\Psi_n[\mathbf{1}]$ can be approximated by  Monte-Carlo simulation. 

In order to calculate the eigenfunction, one can manipulate the following asymptotic.
\[
\langle \tilde\varphi, g\rangle \varphi(r, \upsilon) = \lim_{n \to \infty}\mathbb{E}_{\delta_{(r, \upsilon)}}\left[\frac{1}{n}\sum_{m = 1}^n k_*^{-m}\langle \mathcal{X}_m, g\rangle \right].
\]
Varying the test function $g$, while keeping $(r, \upsilon)$ fixed allows us to obtain estimates for $\tilde\varphi$, whereas varying the initial configuration $(r, \upsilon)$ and keeping the test function fixed allows us to estimate $\varphi$.

Once again, the expectation can be replaced by a Monte-Carlo approximation. 
\bigskip

We refer the reader to~\cite{MCNTE} for a more in-depth discussion of Monte Carlo algorithms based on the above asymptotics, as well as a complexity analysis of their methods. Although the algorithms and efficiency results given in~\cite{MCNTE} are for time-eigenvalues, cf. \eqref{lambda}, it is straightforward to see how they may be adapted to fit the current situation (as well as their complexity). Of course, problems such as burn-in and inefficiencies that were encountered in~\cite{MCNTE} will still be present in the stationary case. We hope to carry out more formal work on this in the future. 

\section*{Acknowledgements}  We are indebted to Geoff Dobson and Paul Smith from the ANSWERS modelling group at Wood for the extensive discussions as well as hosting at their offices in Dorchester.  AEK and AMGC are supported by EPSRC grant EP/P009220/1. ELH is supported by a PhD scholarship with funding from industrial partner Wood.

\bibliography{references}{}

\begin{thebibliography}{10}

\bibitem{SHP}
R.~J. Brissenden and A.~R. Garlick.
\newblock Biases in the estimation of k-eff and its error by monte carlo
  methods.
\newblock {\em Ann. nucl. Energy}, 13:63--83, 1986.

\bibitem{fbrown}
F.~Brown.
\newblock Fundamentals of monte carlo particle transport.
\newblock {\em Lecture notes for Monte Carlo course}.

\bibitem{CV}
N.~Champagnat and D.~Villemonais.
\newblock Exponential convergence to quasi-stationary distribution and
  {$Q$}-process.
\newblock {\em Probab. Theory Related Fields}, 164(1-2):243--283, 2016.

\bibitem{MCNTE}
A.~M.~G. Cox, S.C. Harris, E.~Horton, A.E. Kyprianou, and M.~Wang.
\newblock Monte carlo methods for the neutron transport equation.
\newblock {\em Working document}, 2018.

\bibitem{MultiNTE}
A.M.G. Cox, S.C. Harris, E.~Horton, and A.E. Kyprianou.
\newblock Multi-species neutron transport equation.
\newblock {\em Preprint}, 2018.

\bibitem{D}
R.~Dautray, M.~Cessenat, G.~Ledanois, P.-L. Lions, E.~Pardoux, and R.~Sentis.
\newblock {\em M\'ethodes probabilistes pour les \'equations de la physique}.
\newblock Collection du Commissariat a l'\'energie atomique. Eyrolles, Paris,
  1989.

\bibitem{DL6}
R.~Dautray and J.-L. Lions.
\newblock {\em Mathematical analysis and numerical methods for science and
  technology. {V}ol. 6}.
\newblock Springer-Verlag, Berlin, 1993.
\newblock Evolution problems. II, With the collaboration of Claude Bardos,
  Michel Cessenat, Alain Kavenoky, Patrick Lascaux, Bertrand Mercier, Olivier
  Pironneau, Bruno Scheurer and R\'emi Sentis, Translated from the French by
  Alan Craig.

\bibitem{EN}
K-J. Engel and R.~Nagel.
\newblock {\em A short course on operator semigroups}.
\newblock Universitext. Springer, New York, 2006.

\bibitem{SNTE-II}
S.C. Harris, E.~Horton, and A.E. Kyprianou.
\newblock Stochastic analysis of the neutron transport equation {II}: Almost
  sure growth.
\newblock {\em Preprint}, 2018.

\bibitem{SNTE}
E.~Horton, A.E. Kyprianou, and D.~Villemonais.
\newblock Stochastic analysis of the neutron transport equation.
\newblock {\em Preprint}.

\bibitem{Kato}
T.~Kato.
\newblock {\em Perturbation Theory for Linear Operators}.
\newblock Springer-Verlag, Berlin, 1976.

\bibitem{Kyp}
A.~E. Kyprianou.
\newblock Martingale convergence and the stopped branching random walk.
\newblock {\em Probab Theory Relat. Fields}, 116(3):405--419, 2000.

\bibitem{LM}
E.~E. Lewis and Jr. W.~F.~Miller.
\newblock {\em Computational Methods of Neutron Transport}.
\newblock Wiley \& Sons, New York, 1984.

\bibitem{Mika}
J.~Mika.
\newblock Existence and uniqueness of the solution to the critical problem in
  neutron transport theory.
\newblock {\em Stud. Math.}, 37:213--225, 1971.

\bibitem{M-K}
M.~Mokhtar-Kharroubi.
\newblock {\em Mathematical topics in neutron transport theory}, volume~46 of
  {\em Series on Advances in Mathematics for Applied Sciences}.
\newblock World Scientific Publishing Co., Inc., River Edge, NJ, 1997.
\newblock New aspects, With a chapter by M. Choulli and P. Stefanov.

\bibitem{RV}
M.~Ribari\v{c} and I.~Vidav.
\newblock Analytic properties of the inverse $a(z)^{-1}$ of an analytic linear
  operator valued function a(z).
\newblock {\em Comm. Pure Appl. Math.}, 11:219--242, 1958.

\bibitem{SG}
F.~Scheben and I.~G. Graham.
\newblock Iterative methods for neutron transport eigenvalue problems.
\newblock {\em SIAM J. Sci. Comput.}, 33(5):2785--2804, 2011.

\bibitem{OpTh}
J.~van Neervan.
\newblock {\em The Asymptotic Behaviour of Semigroups of Linear Operators}.
\newblock Operator Theory Advances and Applications Vol. 88. Birkh\"auser,
  1996.

\bibitem{Vidav}
I.~Vidav.
\newblock Existence and uniqueness of nonnegative eigenfunctions of the
  boltzmann operator.
\newblock {\em J. Math. Anal. App.}, 22:144--155, 1968.

\end{thebibliography}
\bibliographystyle{plain}

\end{document}